\documentclass[12pt]{amsart}
\usepackage{amssymb, latexsym}
\usepackage{enumerate}
\usepackage{amsmath}
\usepackage{dsfont}
\usepackage{a4wide}
\usepackage{mathrsfs}
\usepackage[cp1250]{inputenc}
\usepackage[T1]{fontenc}
\usepackage{tabularx}
\usepackage{multirow}
\usepackage{scalerel}
\usepackage{etoolbox}
\usepackage{comment}
\usepackage{xcolor}
\patchcmd{\section}{\normalfont}{\normalfont\Large}{}{}
\patchcmd{\section}{\scshape}{\bfseries}{}{}
\usepackage{lmodern}
\usepackage{hyperref}
\hypersetup{hypertexnames=false}
\let\originalforall=\forall
\renewcommand{\forall}{\mathop{\vcenter{\hbox{\Large$\originalforall$}}}}

\let\originalexists=\exists
\renewcommand{\exists}{\mathop{\vcenter{\hbox{\Large$\originalexists$}}}}

\usepackage[inline]{enumitem}

\DeclareMathOperator\supp{supp}

\newtheorem{thm}{Theorem}[]
\newtheorem*{thm*}{Theorem}

\newtheorem{lem}[thm]{Lemma}

\date{}
\title{Wiener-Pitt sets for compact Abelian groups}
\author{Przemys{\l}aw Ohrysko}
\address{University of Warsaw, Institute of Mathematics, Banacha 2A, 02-097 Warsaw, Poland}
\email{p.ohrysko@gmail.com}
\author{Tom Sanders}
\address{Institute of Mathematics, University of Oxford, United Kingdom}
\email{tom.sanders@maths.ox.ac.uk}
\author{Micha{\l} Wojciechowski}
\address{Institute of Mathematics, Polish Academy of Sciences, 00-656 Warszawa, Poland}
\email{miwoj.impan@gmail.com}
\begin{document}
\begin{abstract}
Suppose that $G$ is a compact Hausdorff Abelian group. We say $\mu \in M(G)$ is strongly continuous if $|\mu|(x+H)=0$ for any $x \in G$ and any $H \leq G$ that is closed and of infinite index.

We prove that for any sufficiently rapidly decreasing sequence $(a_{n})_{n=1}^{\infty}\in c_{0}(\mathbb{N})$, for every strongly continuous $\mu\in M(G)$ with $\|\mu\| \leq 1$ and $\widehat{\mu}(\widehat{G})\subset \{a_n: n \in \mathbb{N}\}\cup\{0\}$, the measure $\mu\ast\mu$ is absolutely continuous with respect to Haar measure on $G$. This implies that $\mu$ does not exhibit the so-called Wiener-Pitt phenomenon.

The paper is a continuation of investigations started in the article `On the relationships between Fourier-Stieltjes coefficients and spectra of measures' published in 2014.
\end{abstract}
\subjclass[2010]{Primary 43A10; Secondary 43A25.}

\keywords{Natural spectrum, Wiener--Pitt phenomenon, measure algebras.}
\maketitle
\section{Introduction}
Throughout the paper $G$ is a locally compact Hausdorff Abelian group and $\widehat{G}$ denotes its (locally compact Hausdorff Abelian) dual group.

By $M(G)$ we denote the unital Banach algebra of all complex-valued regular Borel measures on $G$ equipped with the total variation norm and the convolution product, and for $\mu \in M(G)$ we write $\sigma(\mu)$ for the spectrum of $\mu$ in $M(G)$.

We write $\Delta(M(G))$ for the Gelfand space of $M(G)$, that is the set of not-identically-zero multiplicative and linear maps $M(G) \rightarrow \mathbb{C}$, equipped with the weak$^*$-topology, and in particular we have
\begin{equation}\label{specchar}
\sigma(\mu)=\{\chi(\mu): \chi \in \Delta(M(G))\}.
\end{equation}
The Fourier-Stieltjes transform of $\mu\in M(G)$ is the function $\widehat{\mu}:\widehat{G}\rightarrow\mathbb{C}$ given by:
\begin{equation*}
\widehat{\mu}(\gamma)=\int_{G}\gamma(-x)d\mu(x)\text{ for }\gamma\in\widehat{G}.
\end{equation*}
This gives rise to the following embedding
\begin{equation}\label{embed}
\widehat{G} \rightarrow \Delta(M(G)); \gamma \mapsto (M(G) \rightarrow \mathbb{C}; \mu \mapsto \widehat{\mu}(\gamma)).
\end{equation}
This is an injection -- a measure is determined by its Fourier-Stieltjes transform -- and when $G$ is discrete it is a surjection, but more generally the image is not even dense in the codomain. Indeed, it was shown \cite{owg} that if $G$ is non-discrete then $\Delta(M(G))$ is not separable, whereas there are non-discrete $G$ with countable dual groups \emph{e.g.\ }$G=\mathbb{T}$.

In view of (\ref{specchar}) and (\ref{embed}) certainly $\widehat{\mu}(\gamma) \in \sigma(\mu)$ for all $\gamma \in \widehat{G}$, and since $\sigma(\mu)$ is compact, we have
\begin{equation*}
\overline{\widehat{\mu}(\widehat{G})}\subset\sigma(\mu).
\end{equation*}
If we have equality in this inclusion then following \cite[p618]{Zafran}, $\mu$ is said to have a \textbf{natural spectrum}.

We write $L^1(G)$ for the closed ideal of $M(G)$ of measures that are absolutely continuous with respect to Haar measure on $G$. Any character on $M(G)$ restricts to a multiplicative and linear map $L^1(G) \rightarrow \mathbb{C}$, and this is either identically zero, or else has the form $\mu \mapsto \widehat{\mu}(\gamma)$ \cite[Theorem 4.2, p88]{fol}. Hence by (\ref{specchar}), if $\mu \in L^1(G)$, then $\sigma(\mu) \subset \widehat{\mu}(\widehat{G})\cup\{0\}$. In this case the Riemann-Lebesgue lemma tells us $0$ is a limit point of $\widehat{\mu}(\widehat{G})$, and we conclude that $\mu$ has natural spectrum.

When $G$ is discrete $L^1(G)=M(G)$, and so all measures have natural spectrum but when $G$ is not discrete there are measures which do not and this is called the \textbf{Wiener-Pitt phenomenon}. The first examples of such measures were identified in \cite{wp}, with the first proof in \cite{schreider}, and alternate approaches in \cite{wil} and \cite{graham}.

Since a measure $\mu \in M(G)$ is uniquely determined by its Fourier-Stieltjes transform it is a long-standing problem to explain how various properties of a measure can be deduced from the behaviour of its Fourier-Stieltjes transform. In this paper we will analyse the case in which we only control the range of the Fourier-Stieltjes transform, not the detailed distribution of the values attained on $\widehat{G}$.

A compact set $K$ is one of the titular Wiener-Pitt sets if whenever $\mu \in M(G)$ has $\widehat{\mu}(\widehat{G}) \subset K$ then $\mu$ has natural spectrum. It is a short exercise\footnote{See \cite[p117]{ow}.} to show that all finite sets are Wiener-Pitt sets, and we are interested in finding infinite Wiener-Pitt sets.

The study of this class of problems was (probably) initiated by Sarnak in \cite{s} for multipliers on $L^{p}$-spaces where $p>1$. The main result from \cite{s} can be vaguely described as: `if the closure of the range of the multiplier is sufficiently small (including all countable sets) then that multiplier has a natural spectrum'. Our problem corresponds to multipliers on $L^1$-spaces, which is different and not covered by Sarnak. On the other hand:
\begin{thm}[{\cite[Theorem 2]{ow}}]\label{thm1}
There exists an uncountable compact set $K$ such that if $\mu\in M(\mathbb{T})$ and $\widehat{\mu}(\mathbb{Z})\subset K$ then $\mu$ has a natural spectrum.
\end{thm}

The proof of this result is restricted to the situation in which the Littlewood conjecture (resolved in \cite{ko} and \cite{mps}) holds true which essentially means that there is no hope in extending these arguments beyond the case of groups having (essentially) torsion-free dual. 

The purpose of this paper is to try to extend Theorem \ref{thm1} to arbitrary compact $G$. Whilst we do establish a result in that generality of group, we have to make do with a countably infinite set and we also need an additional restriction on the measures we consider.  This restriction is that $\mu\in M(G)$ be \textbf{strongly continuous}, which means that $|\mu|(x+H)=0$ for all $x \in G$ and $H \leq G$ closed and of infinite index.
\begin{thm}\label{glow}
There is a function $\delta:(0,1] \rightarrow (0,1]$ such that if $(a_n)_{n=1}^\infty$ is a sequence of positive reals with $a_1 \leq 1$ and $a_{n+1} \leq a_n\delta(a_n)$ for all $n \geq 1$, $G$ is compact, and $\mu\in M(G)$ is strongly continuous with $\widehat{\mu}(\widehat{G})\subset\{a_{n}:n\in\mathbb{N}\}\cup\{0\}$ and $\|\mu\|\leq 1$, then $\mu$ has natural spectrum.
\end{thm}
A direct attempt to transport the arguments of \cite{ow} runs into difficulties because general compact Hausdorff Abelian groups may have infinite closed subgroups of infinite index, and Haar measure on such a subgroup, extended to the whole group by letting it be $0$ elsewhere, is a continuous measure that is not strongly continuous.

The proof of Theorem \ref{glow} under the additional assumption that $\widehat{G}$ is torsion-free is easier and since the more general argument is based on it, we begin in \S\ref{sec.torsion} with a proof of this before moving to the full proof of Theorem \ref{glow} in \S\ref{sec.glow}. 

In fact we show something slightly stronger than Theorem \ref{glow}: Theorem \ref{glowrefined} shows that $\mu$ satisfying the hypotheses of Theorem \ref{glow} has $\mu \ast \mu \in L^1(G)$.  It follows by the Riemann-Lebesgue lemma \cite[Proposition 4.13, p93]{fol} that we can apply \cite[Theorem 3.2(a)]{Zafran} to get that $\chi(\mu)^2=\chi(\mu\ast \mu)=0$ for all those $\chi \in \Delta(M(G))$ that do not arise from (\ref{embed}). It then follows from (\ref{specchar}) that $\mu$ has a natural spectrum.\footnote{This is genuinely weaker: the probability measure $\mu$ assigning mass $1$ to the identity has a natural spectrum but $\mu \ast \mu = \mu\not \in L^1(G)$ for $G$ infinite.} 

We close the introduction by remarking that it may be possible to prove a non-Abelian analogue of Theorem \ref{glow}, where the discrete Abelian group $\widehat{G}$ is replaced by a more general discrete group $\Gamma$, and $\widehat{\mu}$ is replaced by an element of $B(\Gamma)\cap c_{0}(\Gamma)$ (the Rajchman algebra).

\section{The torsion-free case}\label{sec.torsion}

In this subsection we will present a proof of Theorem \ref{glow} for the case when $\widehat{G}$ is additionally assumed to be torsion-free. This proof will also be the backbone of the argument in the more general case.

We start by recording a general tool we need in both cases: the so-called uniform invariant approximation property for compact Abelian groups. This was introduced by Bo{\. z}ejko and Pe{\l}czy{\' n}ski in \cite{bp} with an improvement due to Bourgain (see the original work \cite{b} and for a more general result see \cite{o}).
\begin{thm}[BPB construction]\label{BPB}
There is an absolute constant $C>1$ such that for any $\varepsilon\in (0,1]$, if $G$ is compact, and $\Lambda\subset\widehat{G}$ is finite, then there is a trigonometric polynomial $f$ with the properties:

\begin{enumerate*}[label=(\roman*)]
  \item $\widehat{f}(\lambda)=1$ for all $\lambda\in\Lambda$;
  \item $\|f\|_{1}\leq 1+\varepsilon$; and
  \item\label{BPB3} $|\supp \widehat{f}| \leq\left(\frac{C}{\varepsilon}\right)^{2|\Lambda|}$.
\end{enumerate*}
\end{thm}
The notation $\widehat{f}$ here denotes the Fourier-Stieltjes transform of $fd\mu_G$ where $\mu_G$ is Haar probability measure on $G$, and $\|\cdot \|_1$ denotes the $L^1$-norm of $f$ against $\mu_G$.

We shall also need tools to understand strongly continuous measures, and for this purpose we have:
\begin{thm}[{\cite[Theorem 1.4.1, p18]{grmc}}]\label{pom}
There is an absolute constant $C' >1$ such that if $G$ is compact, $r > 2$ is an integer, $\mu\in M(G)$ is strongly continuous with $\|\mu\|<\frac{\sqrt{r}}{4}$, $Q:=\{\gamma\in\widehat{G}:|\widehat{\mu}(\gamma)|\geq 1\}$, and $|\widehat{\mu}(\gamma)|\leq \exp(-r)$ for $\gamma\notin Q$; then $Q$ is finite, and if $\widehat{G}$ is torsion-free then $|Q| \leq \exp(C'r^2 \log r)$. 
\end{thm}

Our main theorem in the case when $G$ has a torsion-free dual is the following:

\begin{thm}\label{najp}
There is an absolute constant $C''>1$ such that if $(a_{n})_{n=1}^{\infty}$ is a sequence of positive reals with $a_1 \leq 1$ and $a_{n+1} \leq a_n\exp(-\exp(C''a_n^{-6}))$ for all $n\geq 1$, then for $G$ compact with torsion-free dual, and $\mu\in M(G)$ strongly continuous with $\widehat{\mu}(\widehat{G})\subset\{a_{n}:n\in\mathbb{N}\}\cup\{0\}$ and $\|\mu\|\leq 1$, we have $\mu\ast\mu\in L^{1}(G)$.
\end{thm}
\begin{proof}
Let $C'':=4C+324C'$ where $C$ and $C'$ are (here and throughout this argument) as in Theorems \ref{BPB} and \ref{pom} respectively, and suppose that $(a_n)_n$ and $\mu$ satisfy the hypotheses of the theorem, so in particular $a_n \leq 2^{-(n-1)}$ for all $n \geq 1$.

By Theorem \ref{pom} applied to $a_n^{-1}\mu$ with $r=\lceil 17a_n^{-2}\rceil$ we have
\begin{equation*}
\left|\bigcup_{k=1}^n{A_k}\right| \leq  \exp(C'\lceil 17a_n^{-2}\rceil^{3}) \text{ where } A_k:=\{\gamma \in \widehat{G}: |\widehat{\mu}(\gamma)| = a_k\},
\end{equation*}
since $a_{n+1} \leq a_n\exp(-\lceil 17a_n^{-2}\rceil)$. For each $n \geq 1$ define $\mu_n \in M(G)$ by Fourier inversion such that $\widehat{\mu_n}=1_{A_n}$ and put $\nu_n:=\sum_{k=1}^{n}{a_k\mu_k}$ so that the support of $\widehat{\nu_n}$ is $\bigcup_{k=1}^n{A_k}$.

For each $n \geq 1$ apply the BPB construction (Theorem \ref{BPB} with $\varepsilon=1$) to the sets $\supp \widehat{\nu_n}$ to get trigonometric polynomials $f_n$ such that

\begin{enumerate*}[label=(\roman*)]
  \item $\widehat{f_n}(\gamma)=1$ for all $\gamma\in\supp\widehat{\nu_{n}}$;
  \item $\|f_n\|_{1}\leq 2$; and
  \item $|\supp \widehat{f_n}| \leq C^{2|\supp \widehat{\nu_n}|}$.
\end{enumerate*}

By Fourier inversion we have
\begin{align}\label{torsionuse}
\|\mu \ast f_n - \nu_n\| & \leq \left\|\frac{d\mu \ast f_n}{d\mu_G} - \frac{d\nu_n}{d\mu_G}\right\|_\infty\\ \nonumber & \leq\sum_{ \gamma\not \in \bigcup_{k=1}^n{A_k}}{|\widehat{\mu}(\gamma)\widehat{f_n}(\gamma)|}\\ \nonumber & \leq  \|f_n\|_1a_{n+1}|\supp \widehat{f_n}|\leq 2a_{n+1}C^{2\exp(C'\lceil 17a_n^{-2}\rceil^{3})} \leq 1.
\end{align}
Since $\|\mu\ast f_{n}\|\leq \|\mu\|\cdot \|f_{n}\|_1\leq 2$, it follows by the triangle inequality that $\|\nu_n\| \leq 3$ for $n \geq 1$ (and trivially for $n=0$ since $\nu_0=0$). By the triangle inequality again
\begin{equation*}
a_n\|\mu_n\| =\|\nu_n-\nu_{n-1}\| \leq \|\nu_n\|+\|\nu_{n-1}\| \leq 6 \text{ for }n \geq 1.
\end{equation*}
Finally, let $\rho_n:=\sum_{k=1}^n{a_k^2\mu_k}$.  Then
\begin{equation*}
\|\rho_n-\rho_m\| \leq \sum_{k=m+1}^n{a_k^2\|\mu_k\|} \leq 6\sum_{k=m+1}^n{a_k}\leq \frac{6}{2^m} \text{ for }n>m.
\end{equation*}
It follows that $(\rho_n)_n$ is Cauchy. Since their Fourier-Stieltjes transforms have finite support, the measures $\rho_n$ are also in $L^1(G)$, which is a closed ideal in $M(G)$ and $M(G)$ is complete, so $\rho_n \rightarrow \rho$ for some $\rho \in L^1(G)$.

Since the $A_k$s are pairwise disjoint we have $\widehat{\rho_n}(\gamma) \rightarrow \widehat{\mu}(\gamma)^2$ for $\gamma \in \widehat{G}$ and hence $\widehat{\rho}(\gamma)=\widehat{\mu}(\gamma)^2$ for all $\gamma \in \widehat{G}$, and by the uniqueness of the Fourier transform $\rho = \mu \ast \mu$ and the result is proved since $\rho \in L^1(G)$.
\end{proof}
To extend the proof above to groups with torsion in the dual we need a bound like (\ref{torsionuse}) without necessarily having a bound on the size of $\supp \widehat{f_n}$ in terms of $a_n$. This is the main work of the next section.

\section{The general case}\label{sec.glow}

The main work of this section is to prove the following theorem which addresses the challenge raised at the end of the last section. After we have done that, a similar argument to the proof of Theorem \ref{najp} yields Theorem \ref{glowrefined} (and hence Theorem \ref{glow} per the discussion at the end of the introduction).

To state the theorem we need some notation: For $x\in G$ we write $n_{x}$ for an integer minimising $|f(x)-n_x|$ over all integers, and put
\begin{equation*}
d(f,\mathbb{Z}):=\sup\{|f(x)-n_{x}|:x\in G\}.
\end{equation*}
 If $d(f,\mathbb{Z})<\frac{1}{2}$, then every $n_{x}$ is uniquely defined and we are allowed to form a function $f_{\mathbb{Z}}:G\rightarrow\mathbb{Z}$ by putting $f_{\mathbb{Z}}(x):=n_{x}$ and in particular
 \begin{equation*}
 d(f,\mathbb{Z})=\|f-f_\mathbb{Z}\|_\infty.
 \end{equation*}

\begin{thm}\label{skn}
There is a function $\delta':(0,1]\times (0,1]\times [1,\infty) \rightarrow (0,1/2)$ such that if $M\geq 1$ and $\varepsilon_1,\varepsilon_2\in (0,1]$, $G$ is compact, $\mu \in M(G)$ has real Fourier-Stieltjes transform with $\|\mu\| \leq M$, $S:=\supp \left(\widehat{\mu}\right)_{\mathbb{Z}}$ is finite, and $d(\widehat{\mu},\mathbb{Z}) \leq \delta'(\varepsilon_1,\varepsilon_2,M)$, then the measure $\mu_\mathbb{Z}$ defined by $d\mu_\mathbb{Z}(x)=\sum_{\gamma \in S}{(\widehat{\mu})_\mathbb{Z}(\gamma)\gamma(x) d\mu_G(x)}$ has $\|\mu_{\mathbb{Z}}\|\leq (1+\varepsilon_{1})\|\mu\|+\varepsilon_{2}$.
\end{thm}
The function $\delta'$ above depends on three parameters, the first two of which represent `errors' -- one multiplicative and one additive.  These could be combined since we have not paid attention to the quantitative aspects of this work, but we have kept them separate to aid understanding in the proofs that follow, as these errors occur in different ways.

We shall prove Theorem \ref{skn}, by first proving it for finite groups and then applying a standard limiting argument. To prove the finite version we need to record some more tools which will be more familiar to those working in `discrete analysis'.

First, the process of passing to the nearest integer respects algebraic properties if the error is small enough:
\begin{lem}\label{aproxadd}
Suppose that $d(f,\mathbb{Z}),d(g,\mathbb{Z}),d(h,\mathbb{Z})<\frac{1}{3}$ and $f=g+h$. Then the functions $f_\mathbb{Z},g_\mathbb{Z},h_\mathbb{Z}$ are well-defined, and $ f_\mathbb{Z}=g_\mathbb{Z}+h_\mathbb{Z}$.
\end{lem}
\begin{proof}
Since $f(x)-(g(x)+h(x))=0$, by the triangle inequality
\begin{align*}
|f_\mathbb{Z}(x)-(g_\mathbb{Z}(x)+h_\mathbb{Z}(x))| & = |f(x)-f_\mathbb{Z}(x) - ((g(x)-g_\mathbb{Z}(x)) + (h(x)-h_\mathbb{Z}(x)))|\\ & \leq d(f,\mathbb{Z})+d(g,\mathbb{Z})+d(h,\mathbb{Z})<1.
\end{align*}
The result is proved since the left hand side is an integer and so it must be $0$.\end{proof}
\subsection*{Annihilators} For $G$ finite the annihilator of $K \leq G$ is
\begin{equation*}
K^\perp:=\{\gamma \in \widehat{G}:\gamma(x)=1\text{ for all }x \in K\}.
\end{equation*}
We write $m_K$ for the uniform probability measure on $G$ with support in $K$ \emph{i.e.\ }
\begin{equation*}
m_K(A)=|A\cap K|/|K|\text{ for all }A \subset G,
\end{equation*}
so in particular $m_G=\mu_G$.
\begin{lem}\label{anncount}
Suppose that $G$ is finite and $K \leq G$. Then $\widehat{m_K}=1_{K^\perp}$; $|K^\perp|\mu_G(K)=1$; and $f:G \rightarrow \mathbb{C}$ is constant on cosets of $K$ if and only if $\supp \widehat{f} \subset K^\perp$. 
\end{lem}
\begin{proof}
We prove the last part first: the `if' direction is by Fourier inversion: if $k \in K$ and $x \in G$ then
\begin{equation*}
f(x+k)=\sum_{\gamma \in \widehat{G}}{\widehat{f}(\gamma)\gamma(x+k)} = \sum_{\gamma \in K^\perp}{\widehat{f}(\gamma)\gamma(x+k)} = \sum_{\gamma \in K^\perp}{\widehat{f}(\gamma)\gamma(x)}=f(x).
\end{equation*}
Conversely, for `only if', given $k \in K$ we have
\begin{align*}
\gamma(k)\widehat{f}(\gamma)& =\int{f(x)\gamma(k-x)d\mu_G(x)} \\ &= \int{f(y+k)\gamma(-y)d\mu_G(y+k)}=\int{f(y)\gamma(-y)d\mu_G(y)}=\widehat{f}(\gamma).
\end{align*}
If $\gamma \not \in K^\perp$ then we can take $k \in K$ with $\gamma(k) \neq 1$, and hence $\widehat{f}(\gamma)=0$. The last part is proved.

If $\gamma \in K^\bot$ then certainly
\begin{equation*}
\widehat{m_K}(\gamma)=\int{\gamma(-x)dm_K(x)}=\int{dm_K(x)}=1,
\end{equation*}
and since $\widehat{m_K}=\mu_G(K)^{-1}\widehat{1_K}$, and $1_K$ is constant on cosets of $K$ we have $\widehat{m_K}(\gamma)=0$ for $\gamma \not \in K^\bot$; in summary $\widehat{m_K}=1_{K^\perp}$. Finally, by Parseval's Theorem we have
 \begin{equation*}
 |K^\perp|\mu_G(K)^2 = \sum_{\gamma \in \widehat{G}}{|\widehat{1_K}(\gamma)|^2} = \int{1_K^2d\mu_G}=\mu_G(K)
 \end{equation*}
completing the lemma when we divide by $\mu_G(K)>0$.
\end{proof}

\subsection*{The quantitive idempotent theorem} Cohen's idempotent theorem \cite{c} describes the structure idempotents in $M(G)$. We shall use a quantitive version from \cite{gs} which is stated dually, for the algebra norm. This is the norm
\begin{equation}\label{dual}
\|f\|_{A(G)}:=\sum_{\gamma\in\widehat{G}}{|\widehat{f}(\gamma)|} = \sup\left\{\left|\sum_{\gamma \in \widehat{G}}{\widehat{f}(\gamma)\overline{\psi(\gamma)}}\right|: \|\psi\|_\infty \leq 1\right\}.
\end{equation}
 In fact the proof proceeds for the wider class of almost-integer-valued functions, and the next result can be read out of the arguments in \cite{gs}, but is also an immediate consequence of \cite[Theorem 10.1]{s2}.
\begin{thm}\label{qit} There are functions $\delta'':[1,\infty) \rightarrow (0,1/2)$ and $F'':[1,\infty) \rightarrow \mathbb{N}$ such that if $M \geq 1$, $G$ is finite, and $f:G \rightarrow \mathbb{R}$ has $\|f\|_{A(G)} \leq M$ and $d(f,\mathbb{Z}) \leq \delta''(M)$, then there are $l \leq 2M$, $H_1,\dots,H_l \leq G$, and functions $g_1,\dots,g_l:G \rightarrow \mathbb{Z}$ such that $f_\mathbb{Z} = g_1+\cdots +g_l$, and for all $1 \leq i \leq l$, $g_i$ is constant on cosets of $H_i$, is supported on at most $F''(M)$ cosets of $H_i$, and $\|g_i\|_\infty \leq F''(M)$.
\end{thm}
In fact we shall need a corollary of Theorem \ref{qit} which we record below in Lemma \ref{corkey}. This corollary is more naturally proved directly using the methods of \cite{gs}, and doing that would give better bounds, but at the expense of making this paper longer and less self-contained.
\begin{lem}\label{corkey}
There are functions $\delta':[1,\infty) \rightarrow (0,1/4)$ and $F':[1,\infty)\times (0,1/4] \rightarrow \mathbb{N}$ such that if $M \geq 1$, $G$ is finite, $f:G \rightarrow \mathbb{R}$ has $\|f\|_{A(G)} \leq M$ and $d(f,\mathbb{Z}) \leq \delta'(M)$, then either $f_\mathbb{Z}\equiv 0$, or for each $\eta \in (0,1/4]$ there is a subgroup $K(\eta) \leq G$ such that
\begin{enumerate}[label=(\alph*)]
\item\label{p1} $d(f\ast m_{K(\eta)},\mathbb{Z})\leq d(f,\mathbb{Z})+\eta$ whenever $0<\eta\leq \frac{1}{4}$;
\item\label{p2} $\|f\ast m_{K(\eta)}\|_{\infty}\geq\frac{1}{2}$ whenever $0<\eta\leq \frac{1}{4}$;
\item\label{p3} $\supp(f\ast m_{K(\eta)})_{\mathbb{Z}}$ is a union of at most $F'(M,\eta)$ cosets of $K(\eta)$, whenever $0<\eta\leq \frac{1}{4}$;
\item\label{p4} $K(\eta') \leq K(\eta)$ whenever $0<\eta' \leq \eta \leq \frac{1}{4}$.
\end{enumerate}
\end{lem}
\begin{proof}
Mindful of where the proof will go, we take $\delta'(M):=\delta''(M)/2$ and
 \begin{equation}\label{FMeta}
     F'(M,\eta):= \left \lceil \max{\left\{ \eta^{-(l-1)}l^l F''(M)^{2l-1}:1 \leq l \leq 2M\right\}}\right\rceil
 \end{equation}
 where $\delta''$ and $F''$ are the functions in Theorem \ref{qit}.
 
 Now, suppose $M \geq 1$, $G$ is finite, $f:G \rightarrow \mathbb{R}$ has $\|f\|_{A(G)} \leq M$ and $\delta:=d(f,\mathbb{Z}) \leq \delta'(M)$. Since $\delta'(M)\leq \delta''(M)$ Theorem \ref{qit} gives us $l \leq 2M$, $H_1,\dots, H_l \leq G$, and functions $g_1,\dots,g_l:G \rightarrow \mathbb{Z}$ such that $f_\mathbb{Z}=g_1+\cdots + g_l$ where for each $i \in \{1,\dots,l\}$, the function $g_i$ is constant on cosets of $H_i$, is supported on at most $F''(M)$ cosets of $H_i$, and has $\|g_i\|_\infty \leq F''(M)$.
 
Suppose $f_\mathbb{Z} \not \equiv 0$, then we may also assume that no proper non-empty subset of the $g_i$s sums to be identically $0$ -- if there were such a sum we could just remove those $H_i$s and $g_i$s.

It will be notationally convenient to reorder the $H_i$s so that for all $k \in \{1,\dots,l\}$ we have
\begin{equation}\label{reorder}
|H_1\cap \dots \cap H_{k-1} \cap H_{k}|  \geq |H_1\cap \dots \cap H_{k-1} \cap H_{j}| \text{ for all }j \in\{ k,\dots,l\}
\end{equation}
with the convention that if $k=1$ (\ref{reorder}) just says that $|H_1| \geq |H_j|$ for all $j \in \{1,\dots,l\}$. This reordering can just be done greedily, and of course need not be unique.

For $k \in \{1,\dots,l\}$ write $K_k:=H_1\cap \dots \cap H_k$, so that $K_1 \geq \dots \geq K_l$.  For each $\eta \in (0,\frac{1}{4}]$ let
\begin{equation}\label{neta}
  k(\eta):=\min{\left\{k \in \{1,\dots,l-1\}:  \frac{|K_{k} \cap H_{k+1}|}{|K_{k}|}\leq \frac{\eta}{lF''(M)^2}\right\}\cup \{l\}}.
\end{equation}
With this we put $K(\eta):=K_{k(\eta)}$, so that if $\eta'\leq \eta$ then $k(\eta) \leq k(\eta')$, and so $K(\eta') \leq K(\eta)$. This establishes property \ref{p4} of the $K(\eta)$s. To show \ref{p1}, \ref{p2}, and \ref{p3} we shall simplify notation and write $k:=k(\eta)$ and $K:=K(\eta)$.

First we estimate the size of $K$: By (\ref{neta}), for $j< k$,
\begin{equation*}
\frac{|K_{j+1}|}{|K_j|}=\frac{|K_{j}\cap H_{j+1}|}{|K_j|} >\frac{\eta}{lF''(M)^2}.
\end{equation*}
Hence by telescoping,
 \begin{equation}\label{lowerK}
    |K|=|K_{k}|=\frac{|K_k|}{|K_{k-1}|}\cdot\ldots\cdot \frac{|K_2|}{|K_1|}\cdot |K_1| >\left(\frac{\eta}{lF''(M)^2}\right)^{k-1}|K_{1}| \geq  \left(\frac{\eta}{lF''(M)^{2}}\right)^{l-1}|H_{1}|.
 \end{equation}

Now to our main estimates: $K \leq H_j$ for all $j \leq k$ and, as $g_{j}$ is constant on cosets of $H_{j}$ for all $j$, we have
\begin{equation}\label{smooth}
\sum_{j \leq k}{g_j \ast m_K}= \sum_{j \leq k}{g_j}.
\end{equation}
Write $\mathcal{W}_j$ for the set of cosets of $H_j$ on which $g_j$ is supported. Then
\begin{equation*}
|g_j(y)| \leq \sum_{x+H_j \in \mathcal{W}_j}{\|g_j\|_\infty\cdot 1_{x+H_j}(y)} \text{ for all } y \in G,
\end{equation*}
and so by the triangle inequality and the bounds $\|g_j\|_\infty \leq F''(M)$ and $|\mathcal{W}_j|\leq F''(M)$, we have
\begin{align}
\nonumber \left \|\sum_{j >k}{g_j \ast m_K}\right\|_\infty & \leq \sum_{j >k}{\|g_j\|_\infty \cdot |\mathcal{W}_j|\cdot \max_{x \in G}{\|1_{x+H_j} \ast m_K\|_\infty}}\\ \nonumber & = \sum_{j >k}{\|g_j\|_\infty \cdot |\mathcal{W}_j|\cdot \|1_{H_j} \ast m_K\|_\infty}\\ \label{last}& \leq F''(M)^2\cdot \sum_{j>k}{\frac{|K \cap H_j|}{|K|}}.
\end{align}
The equality $\|1_{H_j} \ast m_K\|_\infty=\frac{|K \cap H_j|}{|K|}$ follows from the fact that the intersection of two cosets of two subgroups is either empty or a coset of the intersection of these subgroups.

Now, if $k<l$ then by (\ref{last}), and then (\ref{reorder}) and (\ref{neta}) we have
\begin{equation}\label{smooth2}
 \left \|\sum_{j >k}{g_j \ast m_K}\right\|_\infty  \leq F''(M)^2\cdot \sum_{j>k}{\frac{|K \cap H_j|}{|K|}}\leq  F''(M)^2\cdot l \cdot \frac{|K \cap H_{k+1}|}{|K|} \leq \eta.
\end{equation}
Of course, if $k=l$ the sum on the left is empty and so the inequality certainly also holds.

By the triangle inequality and (\ref{smooth}) and (\ref{smooth2})
\begin{align*}
\left \|f \ast m_K -  \sum_{j \leq k}{g_j}\right\|_\infty & \leq \|(f-f_\mathbb{Z})\ast m_K\|_\infty + \left\|f_\mathbb{Z}\ast m_K - \sum_{j  \leq k}{g_j}\right\|_\infty\\ & \leq \delta+ \left\|\left(g_1+\cdots + g_l\right)\ast m_K - \sum_{j  \leq k}{g_j\ast m_K}\right\|_\infty \leq \delta+\eta,
\end{align*}
whence $d(f\ast m_K,\mathbb{Z}) \leq \delta + \eta$, and \ref{p1} is established. Since $\delta +\eta \leq \delta'(M)+1/4<1/2$, we may also conclude that
\begin{equation}\label{dec}
(f \ast m_K)_\mathbb{Z} = \sum_{j  \leq k}{g_j}.
\end{equation}
The right hand side is not identically $0$ and hence $\|f \ast m_K\|_\infty \geq \frac{1}{2}$ \emph{i.e.\ }\ref{p2} is shown. Finally, the right hand side of (\ref{dec}) is non-zero on at most 
\begin{equation*}
\sum_{j  \leq k}{|\mathcal{W}_j| \cdot \frac{|H_{j}|}{|K \cap H_{j}|}}\leq F''(M)\cdot l \cdot \frac{|H_{1}|}{|K|} \leq F'(M,\eta)
\end{equation*}
 cosets of $K$. The first inequality is valid since $|H_j| \leq |H_1|$ for all $j$ and $|K \cap H_j| = |K|$ for all $j \leq k$; the second inequality follows from (\ref{lowerK}) and the definition of $F'(M,\eta)$ in (\ref{FMeta}). This gives \ref{p3} and the proof of the lemma is complete.
\end{proof}

\begin{thm}\label{tws}
There is a function $\delta:(0,1]\times (0,1]\times [1,\infty) \rightarrow (0,1/2)$ such that if $M\geq 1$ and $\varepsilon_1,\varepsilon_2\in (0,1]$, $G$ is finite, $f:G\rightarrow\mathbb{R}$ has $\|f\|_{A(G)}\leq M$ and $d(f,\mathbb{Z})\leq\delta(\varepsilon_1,\varepsilon_2,M)$, then $\|f_{\mathbb{Z}}\|_{A(G)}\leq (1+\varepsilon_1)\|f\|_{A(G)}+\varepsilon_2$.
\end{thm}
\begin{proof}
For $m \geq 1$ a natural  number write $A_m:=(0,1]\times (0,1]\times [m/2,(1+m)/2)$. We shall define $\delta_m:A_m \rightarrow (0,1/2)$ such that
\begin{equation}\tag{$P_m$} \label{inductivestep}
\parbox{0.9\linewidth}{%
for all $(\varepsilon_1,\varepsilon_2,M) \in A_m$, $G$ finite, $f:G\rightarrow\mathbb{R}$ with $\|f\|_{A(G)}\leq M$ and $d(f,\mathbb{Z})\leq\delta_m(\varepsilon_1,\varepsilon_2,M)$, we have $\|f_{\mathbb{Z}}\|_{A(G)}\leq (1+\varepsilon_1)\|f\|_{A(G)}+\varepsilon_2$.
}
\end{equation}
The function $\delta$ is then simply the functions $\delta_m$ pieced together on $\bigcup_{m\geq 2}{A_m}$. (We exclude $A_1$ in the domain of $\delta$ in the statement of the theorem; it is convenient here as a base to start an induction.)

For $m=1$, set $\delta_1(\varepsilon_1,\varepsilon_2,M):=(1-M)/2 \in (0,1/2)$ and suppose that $(\varepsilon_1,\varepsilon_2,M) \in A_1$, $G$ is finite, $f:G\rightarrow\mathbb{R}$ has $\|f\|_{A(G)}\leq M$ and $d(f,\mathbb{Z})\leq\delta_1(M,\varepsilon_1,\varepsilon_2)$. If $f_\mathbb{Z}\not \equiv 0$ then
\begin{equation*}
M \geq \|f\|_{A(G)} \geq \|f\|_\infty \geq \|f_\mathbb{Z}\|_\infty - d(f,\mathbb{Z}) \geq 1-(1-M)/2,
\end{equation*}
and so $M \geq 1$ --  a contradiction to the fact $M \in [m/2,(1+m)/2)$. Hence $f_\mathbb{Z} \equiv 0$, and so
\begin{equation}\label{zerodone}
\|f_\mathbb{Z}\|_{A(G)}=0 \leq (1+\varepsilon_1)\|f\|_{A(G)}+\varepsilon_2
\end{equation}
and ($P_1$) is proved.

For what remains it will be useful to write $F'':[1,\infty) \times (0,1/4]\times (0,1] \rightarrow \mathbb{N}$ for
\begin{equation*}
F''(M,\eta,\varepsilon_1):=\left\lceil (1+\varepsilon_1)\left(\frac{C}{\varepsilon_1}\right)^{2F'(M,\eta)}\right \rceil
\end{equation*}
where $F'$ is as in the statement of Lemma \ref{corkey} and $C$ is as in Theorem \ref{BPB}.

Suppose that $m \geq 1$ is a natural number and assume that $\delta_m$ has been defined satisfying ($P_m$). To define $\delta_{m+1}$ we introduce some auxiliary variables and functions. For $(\varepsilon_1,\varepsilon_2,M) \in A_{m+1}$ note that $M \geq 1$ and $\left(\varepsilon_1,\frac{\varepsilon_2}{2},M-\frac{1}{2}\right) \in A_m$. Set $l:=\lceil 3M\rceil $ and
 \begin{equation}\label{eta0}
\eta_0:=\frac{1}{3}\delta_m\left(\varepsilon_1,\frac{\varepsilon_2}{2},M-\frac{1}{2}\right)\in (0,1/4]
\end{equation}
and 
\begin{equation}\label{etai+1}
\eta_{i+1}:=\min\left\{\eta_i,\frac{\varepsilon_{2}}{4}F''(M,\eta_{i},\varepsilon_{1})^{-1}\right\}\in (0,1/4]
\text{ for }0 \leq i <l.
\end{equation}
With these define $\delta_{m+1}(\varepsilon_1,\varepsilon_2,M):=\min\{\delta'(M),\eta_l\}$ where $\delta'$ is as in the statement of Lemma \ref{corkey}.

Now, suppose that $(\varepsilon_1,\varepsilon_2,M) \in A_{m+1}$, $G$ is finite, $f:G\rightarrow\mathbb{R}$ with $\|f\|_{A(G)}\leq M$ and $d(f,\mathbb{Z})\leq\delta_{m+1}(\varepsilon_1,\varepsilon_2,M)$. Let $l:=\lceil 3M\rceil$ and define $\eta_i$ for $0 \leq i \leq l$ as in (\ref{eta0}) and (\ref{etai+1}).

Since $d(f,\mathbb{Z}) \leq \delta_{m+1}(\varepsilon_1,\varepsilon_2,M) \leq \delta'(M)$ (where $\delta'$ is the function from Lemma \ref{corkey}), we may apply Lemma \ref{corkey} (either to get $f_\mathbb{Z}\equiv 0$ in which case we are done as in (\ref{zerodone}) or) to get a group $K(\eta)$ for each $\eta \in (0,1/4]$ such that these groups satisfy properties \ref{p1}--\ref{p4} of Lemma \ref{corkey}. In particular, since $\frac{1}{6}> \eta_{0}\geq \eta_{1} \geq \ldots\geq \eta_{l}>0$,
\begin{enumerate}[label=(\alph*)]
\item\label{prop1} $d(f\ast m_{K(\eta_i)},\mathbb{Z})\leq \delta_{m+1}(\varepsilon_1,\varepsilon_2,M) +\eta_i\leq 2\eta_i$ for all $0 \leq i \leq l$;
\item\label{prop2} $\|f\ast m_{K(\eta_i)}\|_{\infty}\geq\frac{1}{2}$ for all $0 \leq i \leq l$;
\item\label{prop3} $\supp(f\ast m_{K(\eta_i)})_{\mathbb{Z}}$ is a union of $F'(M,\eta_i)$ cosets of $K(\eta_i)$ for all $0 \leq i \leq l$;
\item\label{prop4}  $K(\eta_0) \geq \dots \geq K(\eta_l)$.
\end{enumerate}
In view of the definition of annihilators, the nesting in \ref{prop4} reverses when we take annihilators \emph{i.e.\ }$K(\eta_{l})^{\bot}\geq\ldots \geq K(\eta_{0})^{\bot}$.  Thus, since $\widehat{m_{K}}=1_{K^{\bot}}$ for any $K\leq G$, we have
\begin{equation*}
\sum_{i=0}^{l-1}\|f\ast m_{K(\eta_{i})}-f\ast m_{K(\eta_{i+1})}\|_{A(G)}=\sum_{i=0}^{l-1}{\sum_{\gamma \in K(\eta_{i+1})^\perp \setminus K(\eta_{i})^\perp}{|\widehat{f}(\gamma)|}}\leq \|f\|_{A(G)}.
\end{equation*}
By the pigeonhole principle there is some $0\leq i_{0}<l$ such that
\begin{equation}\label{jt}
\|f\ast m_{K(\eta_{i_{0}})}-f\ast m_{K(\eta_{i_{0}+1})}\|_{\infty}\leq\|f\ast m_{K(\eta_{i_{0}})}-f\ast m_{K(\eta_{i_{0}+1})}\|_{A(G)}\leq\frac{M}{l}<\frac{1}{3}.
\end{equation}
By the triangle inequality and \ref{prop1} we have
\begin{align*}
\|(f\ast m_{K(\eta_{i_{0}})})_{\mathbb{Z}}-(f\ast m_{K(\eta_{i_{0}+1})})_{\mathbb{Z}}\|_\infty & \leq \|f\ast m_{K(\eta_{i_{0}})}-f\ast m_{K(\eta_{i_{0}+1})}\|_{\infty}\\ & \qquad  +d(f\ast m_{K(\eta_{i_{0}})},\mathbb{Z})+d(f\ast m_{K(\eta_{i_{0}+1})},\mathbb{Z})\\
& <\frac{1}{3} + 2\eta_{i_{0}}+2 \eta_{i_{0}+1} <\frac{1}{3}+\frac{4}{6}=1,
\end{align*}
and so $(f\ast m_{K(\eta_{i_{0}})})_{\mathbb{Z}}=(f\ast m_{K(\eta_{i_{0}+1})})_{\mathbb{Z}}$. Establishing this consequence is the purpose of the sequence of $\eta_i$s we defined earlier; going forward write $K:= K(\eta_{i_{0}})$ and $H:= K(\eta_{i_{0}+1})$, and then what we need (from what we have already shown) is:
\begin{enumerate}[label=(\alph*')]
\item \label{ctd}  $d(f\ast m_{K},\mathbb{Z})=d(f\ast m_H,\mathbb{Z}) \leq 2\eta_{{i_0}+1}$;
\item\label{cte} $\|f\ast m_{H}\|_{\infty}\geq\frac{1}{2}$;
\item \label{ctt} $\supp(f\ast m_{H})_{\mathbb{Z}}=\supp(f\ast m_{K})_{\mathbb{Z}}$ is a union of $F'(M,\eta_{i_{0}})$ cosets of $K$.
\end{enumerate}
Since $G$ is finite, the supremum in (\ref{dual}) is attained by the Heine-Borel Theorem and so (by Fourier inversion, again since $G$ is finite) there exists $\varphi:G\rightarrow\mathbb{C}$ satisfying $\|\widehat{\varphi}\|_{\infty}\leq 1$ and
\begin{equation}\label{wno}
\sum_{\gamma\in \widehat{G}}((f\ast m_{H})_{\mathbb{Z}})\widehat{\phantom{x}}(\gamma)\cdot\overline{\widehat{\varphi}(\gamma)}=\|(f\ast m_{H})_{\mathbb{Z}}\|_{A(G)}.
\end{equation}
Since $(f\ast m_{H})_{\mathbb{Z}}$ is constant on cosets of $K$ the Fourier transform of $(f\ast m_{H})_{\mathbb{Z}}$ is supported on $K^{\bot}$ (by Lemma \ref{anncount})  so we may assume that $\supp\widehat{\varphi}\subset K^{\bot}$.

Let $\mathcal{W}$ be the set of cosets of $K$ on which $(f\ast m_{K})_{\mathbb{Z}}$ is supported, so that by \ref{ctt} we have $|\mathcal{W}| \leq F'(M,\eta_{i_0})$ and $x \in  \supp (f \ast m_H)_\mathbb{Z}$ if and only if $x +K \in \mathcal{W}$. 

 The map
\begin{equation*}
\Phi:G/K \rightarrow \widehat{K^\perp}; x+K \mapsto (K^\perp \rightarrow S^1; \gamma \mapsto \gamma(x))
\end{equation*}
 is a well-defined isomorphism \cite[Theorem 2.1.2]{r} and so by the BPB construction (Theorem \ref{BPB}) applied to the subset $\Lambda:=\{\Phi(W):W \in \mathcal{W})\}$ of $\widehat{K^\perp}$ we get a trigonometric polynomial $g$ on $K^{\bot}$ (which although naturally discrete, is also finite and so compact) such that
 
\begin{enumerate*}[label=(\roman*)]
  \item\label{const} $\widehat{g}(\Phi(W))=1$ for all $W \in \mathcal{W}$;
  \item\label{g1} $\|g\|_{1}\leq 1+\varepsilon_1$; and
  \item\label{g3} $|\supp\widehat{g} |\leq\left(\frac{C}{\varepsilon_1}\right)^{2F'(M,\eta_{i_{0}})}$.
  \end{enumerate*}
  
  Define $\phi:G \rightarrow \mathbb{C}; x \mapsto \varphi(x)\widehat{g}(\Phi(x+K))$. Then we claim the following facts:
\begin{enumerate}[label=(\Alph*)]
  \item\label{k1} $ \langle (f\ast m_{H})_{\mathbb{Z}},\phi\rangle=\|(f\ast m_{H})_{\mathbb{Z}}\|_{A(G)}$.
 \begin{proof}$\widehat{g}(\Phi(x+K))=1$ when $x \in \supp (f \ast m_H)_{\mathbb{Z}}$ (by \ref{const} and the definition of $\mathcal{W}$) and then Parseval's Theorem gives
\begin{equation*}
\langle(f\ast m_{H})_{\mathbb{Z}},\phi\rangle =\langle (f \ast m_H)_{\mathbb{Z}},\varphi\rangle =\sum_{\gamma\in \widehat{G}}((f\ast m_{H})_{\mathbb{Z}})\widehat{\phantom{x}}(\gamma)\cdot\overline{\widehat{\varphi}(\gamma)},
\end{equation*}
and the claim follows by (\ref{wno}).\end{proof}
  \item\label{k2} $\|\widehat{\phi}\|_{\infty}\leq 1+\varepsilon_1$.
 \begin{proof} Since
   \begin{equation*}
   \widehat{g}(\Phi(x+K))=\frac{1}{|K^\perp|}\sum_{\gamma \in K^\perp}{g(\gamma)\Phi(x+K)(-\gamma)}=\frac{1}{|K^\perp|}\sum_{\gamma \in K^\perp}{g(\gamma)\gamma(-x)}
   \end{equation*}
   for all $x \in G$, we have
   \begin{align*}
   \widehat{\phi}(\lambda)& =\frac{1}{|G|}\sum_{x \in G}{\varphi(x)\widehat{g}(\Phi(x+K))\lambda(-x)}\\ &=\frac{1}{|K^\perp|}{\sum_{\gamma \in K^\perp}{g(\gamma)\frac{1}{|G|}\sum_{x \in G}{\varphi(x)\gamma(-x)\lambda(-x)}}} =\frac{1}{|K^\perp|}\sum_{\gamma \in K^\perp}{g(\gamma)\widehat{\varphi}(\gamma+\lambda)}.
   \end{align*}
   It follows by the triangle inequality that $\|\widehat{\phi}\|_\infty \leq \|g\|_1\|\widehat{\varphi}\|_\infty$ which gives the claim by \ref{g1} and design of $\widehat{\varphi}$.\end{proof}
  \item\label{k3} $\supp \widehat{\phi} \subset K^\perp$.
\begin{proof}As noted after (\ref{wno}) we may assume that $\supp \widehat{\varphi} \subset K^\perp$ and so by Lemma \ref{anncount} $\varphi$ is constant on cosets of $K$. The function $G \rightarrow \mathbb{C};x \mapsto \widehat{g}(\Phi(x+K))$ is visibly constant on cosets of $K$, and so $\phi$, which is the product of this function and $\varphi$, is constant on cosets of $K$ and hence, again by Lemma \ref{anncount} but in the other direction, $\supp \widehat{\phi} \subset K^\perp$.\end{proof}
  \item\label{k5} $\|\phi\|_{1}\leq F''(M,\eta_{i_{0}},\varepsilon_{1})$.
\begin{proof} By design, $\supp \phi \subset \{x \in G: \Phi(x+K) \in \supp \widehat{g}\}$, and so by \ref{g3} we have
\begin{equation}\label{firstcalc}
\mu_G(\supp \phi) \leq \mu_G(K) |\supp \widehat{g}|\leq \mu_G(K)\left(\frac{C}{\varepsilon_1}\right)^{2F'(M,\eta_{i_{0}})}.
\end{equation}
By the triangle inequality, Fourier inversion for $\phi$, and the triangle inequality again we have
\begin{equation*}
\|\phi\|_{1} \leq\|\phi\|_{\infty}\cdot \mu_G(\supp\phi) \leq \|\widehat{\phi}\|_\infty\cdot |\supp\widehat{\phi}| \cdot \mu_G(\supp \phi).
\end{equation*}
Hence by \ref{k2}, \ref{k3}, and (\ref{firstcalc}) we have
\begin{equation*}
\|\phi\|_1\leq (1+\varepsilon_1)\cdot |K^\perp|\cdot \mu_G(K)\left(\frac{C}{\varepsilon_1}\right)^{2F'(M,\eta_{i_{0}})}=F''(M,\eta_{i_0},\varepsilon_1)
\end{equation*}
since $|K^\perp|\mu_G(K)=1$ (Lemma \ref{anncount}).
\end{proof}
\end{enumerate}
By duality and then \ref{k2}, and then \ref{k1}, and then the triangle inequality, we have
\begin{align*}
\|f\ast m_{H}\|_{A(G)}& \geq\frac{1}{\|\widehat{\phi}\|_\infty}\left| \sum_{\gamma \in \widehat{G}}{(f \ast m_H)\widehat{\phantom{x}}(\gamma)\overline{\widehat{\phi}(\gamma)}}\right|  \geq\frac{|\langle f\ast m_{H},\phi\rangle |}{1+\varepsilon_{1}}\\ &\qquad \qquad 
=\frac{1}{1+\varepsilon_{1}}\left|\|(f\ast m_{H})_{\mathbb{Z}}\|_{A(G)}+\langle f\ast m_{H}-(f\ast m_{H})_{\mathbb{Z}},\phi\rangle\right|\\ &\qquad \qquad 
\geq \frac{1}{1+\varepsilon_{1}}\left(\|(f\ast m_{H})_{\mathbb{Z}}\|_{A(G)}-d(f\ast m_{H},\mathbb{Z})\|\phi\|_{1}\right).
\end{align*}
Hence by \ref{k5} and the fact that $d(f\ast m_{H},\mathbb{Z})\leq 2\eta_{i_{0}+1}$ (from \ref{ctd}) we get
\begin{equation}\label{wazo}
\|(f\ast m_{H})_{\mathbb{Z}}\|_{A(G)}\leq (1+\varepsilon_{1})\|f\ast m_{H}\|_{A(G)}+2\eta_{i_{0}+1}F''(M,\eta_{i_{0}},\varepsilon_{1}).
\end{equation}
On the other hand, since $\widehat{m_H}=1_{H^\perp}$, we have 
\begin{equation}\label{ju}
\|f\|_{A(G)}=\|f\ast m_{H}\|_{A(G)}+\|f-f\ast m_{H}\|_{A(G)},
\end{equation}
and since $\|f\ast m_{H}\|_{A(G)}\geq\|f\ast m_{H}\|_{\infty}\geq\frac{1}{2}$ (by \ref{cte}) we have
\begin{equation}\label{zalind}
\|f-f\ast m_{H}\|_{A(G)}\leq \|f\|_{A(G)}-\frac{1}{2}\leq M-\frac{1}{2}.
\end{equation}
By the triangle inequality $d(f-f\ast m_{H},\mathbb{Z})\leq 2\eta_{i_{0}+1}+\delta_{m+1}(\varepsilon_1,\varepsilon_2,M) \leq 3\eta_{{i_0}+1}$. Now $\eta_{{i_0}+1} \leq \eta_0 \leq \frac{1}{3}\delta_m(\varepsilon_1,\frac{\varepsilon_2}{2},M-\frac{1}{2})$, and hence
\begin{equation*}
\|f-f\ast m_{H}\|_{A(G)}\leq M-\frac{1}{2} \text{ and } d(f-f\ast m_{H},\mathbb{Z})\leq \delta_m(\varepsilon_1,\frac{\varepsilon_2}{2},M-\frac{1}{2}).
\end{equation*}
It follows by the inductive hypothesis ($P_m$), that 
\begin{equation}\label{tind}
\|(f-f\ast m_{H})_{\mathbb{Z}}\|_{A(G)}\leq (1+\varepsilon_{1})\|f-f\ast m_{H}\|_{A(G)}+\frac{\varepsilon_{2}}{2}.
\end{equation}
Since $d(f,\mathbb{Z}),d(f\ast m_H,\mathbb{Z}),d(f-f\ast m_H,\mathbb{Z})<\frac{1}{3}$, by Lemma \ref{aproxadd}, $f_{\mathbb{Z}}=(f\ast m_{H})_{\mathbb{Z}}+(f-f\ast m_{H})_{\mathbb{Z}}$ so by the triangle inequality
\begin{equation*}
\|f_{\mathbb{Z}}\|_{A(G)}\leq \|(f\ast m_{H})_{\mathbb{Z}}\|_{A(G)}+\|(f-f\ast m_{H})_{\mathbb{Z}}\|_{A(G)},
\end{equation*}
and then using estimates (\ref{wazo}) and (\ref{tind}), followed by the identity (\ref{ju})
\begin{align*}
\|f_{\mathbb{Z}}\|_{A(G)} & \leq (1+\varepsilon_{1})\|f\ast m_{H}\|_{A(G)}+2\eta_{i_{0}+1}F''(M,\eta_{i_{0}},\varepsilon_{1})\\ &\qquad + (1+\varepsilon_{1})\|f-f\ast m_{H}\|_{A(G)}+\frac{\varepsilon_{2}}{2}\\
& =(1+\varepsilon_{1})\|f\|_{A(G)}+2\eta_{i_{0}+1}F''(M,\eta_{i_{0}},\varepsilon_{1})+\frac{\varepsilon_{2}}{2} \leq (1+\varepsilon_{1})\|f\|_{A(G)}+\varepsilon_2
\end{align*}
since $\eta_{i_{0}+1}\leq \frac{1}{4}\varepsilon_{2}F''(M,\eta_{i_{0}},\varepsilon_{1})^{-1}$. This proves ($P_{m+1}$), and the result follows by induction.
\end{proof}

Theorem \ref{tws} can be easily extended to complex-valued functions: Certainly $d(f,\mathbb{Z})\geq\max\{d(\operatorname{Re}f,\mathbb{Z}),\|\mathrm{Im}f\|_{\infty}\}$, and so $f_\mathbb{Z}=(\operatorname{Re} f)_\mathbb{Z}$ while $\|\operatorname{Re}f\|_{A(G)}=\|\frac{f+\overline{f}}{2}\|_{A(G)}\leq \|f\|_{A(G)}$ and so we can apply Theorem \ref{tws} to $\operatorname{Re}f$ to get the result.

Finally we need a standard limiting argument to extend Theorem \ref{tws} to infinite groups.

\begin{thm*}[Theorem \ref{skn}]
There is a function $\delta':(0,1]\times (0,1]\times [1,\infty) \rightarrow (0,1/2)$ such that if $M\geq 1$ and $\varepsilon_1,\varepsilon_2\in (0,1]$, $G$ is compact, $\mu \in M(G)$ has real Fourier-Stieltjes transform with $\|\mu\| \leq M$, $S:=\supp \left(\widehat{\mu}\right)_{\mathbb{Z}}$ is finite, and $d(\widehat{\mu},\mathbb{Z}) \leq \delta'(\varepsilon_1,\varepsilon_2,M)$, then the measure $\mu_\mathbb{Z}$ defined by $d\mu_\mathbb{Z}(x)=\sum_{\gamma \in S}{(\widehat{\mu})_\mathbb{Z}(\gamma)\gamma(x) d\mu_G(x)}$ has $\|\mu_{\mathbb{Z}}\|\leq (1+\varepsilon_{1})\|\mu\|+\varepsilon_{2}$.
\end{thm*}
\begin{proof}
We take $\delta'(\varepsilon_1,\varepsilon_2,M):=\frac{1}{2}\delta(\frac{1}{3}\varepsilon_1,\varepsilon_2,2M)$ where $\delta$ is as in Theorem \ref{tws}.  Apply the BPB construction (Theorem \ref{BPB}) to $S$ to get a trigonometric polynomial $p$ such that:

\begin{enumerate*}[label=(\roman*)]
  \item $\widehat{p}(\gamma)=1$ for $\gamma\in S$;
  \item $\|p\|_{1}\leq 1+\frac{1}{3}\varepsilon_1$; and
  \item $\Lambda:=\supp \widehat{p}$ is finite. 
\end{enumerate*}

Let $\nu:=p \ast \mu$ and note that $\supp \widehat{\nu} \subset \Lambda$ is finite and $\|\nu\|\leq \|p\|_1 \|\mu\| \leq (1+\frac{1}{3}\varepsilon_1)M$. If $\widehat{\nu}(\gamma) \neq 0$ then either $\gamma \in S$ and $\widehat{\nu}(\gamma)=\widehat{\mu}(\gamma)$; or $\gamma \not \in S$ and so $|\widehat{\nu}(\gamma)|=|\widehat{p}(\gamma)||\widehat{\mu}(\gamma)| \leq 2|\widehat{\mu}(\gamma)| \leq 2d(\widehat{\mu},\mathbb{Z})$. In summary, $d(\widehat{\nu},\mathbb{Z}) \leq \delta(\frac{1}{3}\varepsilon_1,\varepsilon_2,2M)$ and $(\widehat{\nu})_\mathbb{Z} = (\widehat{\mu})_\mathbb{Z}$.

Let $\Gamma$ be the subgroup of $\widehat{G}$ generated by $\Lambda$, which is then finitely generated, and so by the structure theorem for finitely generated Abelian groups there is a finite Abelian group $H$ and an isomorphism $\Gamma \rightarrow \widehat{H} \times \mathbb{Z}^d$. 

Writing $K:=\{x \in G: \gamma(x)=1\text{ for all }\gamma \in \Gamma\}$ we have by \cite[Theorem 2.1.2]{r} that the dual of $G/K$ is homeomorphically isomorphic to $\Gamma$, and hence there is a homeomorphic isomorphism $\pi:G/K \rightarrow H \times (\mathbb{R}/\mathbb{Z})^d$. For each $\gamma \in \Gamma$ there is a unique $\tilde{\gamma} \in (H \times (\mathbb{R}/\mathbb{Z})^d)^\wedge$ such that $\gamma=\tilde{\gamma}\circ \pi\circ q$ and there is a unique $(\chi_\gamma ,r_1(\gamma),\dots,r_d(\gamma)) \in \widehat{H} \times \mathbb{Z}^d$ such that
\begin{equation}\label{egam}
\tilde{\gamma}(h,\theta_1,\dots,\theta_d)=\chi_\gamma(h)\exp(2\pi i r_1(\gamma)\theta_1)\cdots \exp(2\pi i r_d(\gamma)\theta_d).
\end{equation}
Define $g,h:\widehat{H} \times \mathbb{Z}_N^d\rightarrow \mathbb{R}$ by
\begin{equation*}
g(\chi,r_1,\dots,r_d):=\sum_{\substack{\gamma \in \Lambda\\ \chi_\gamma=\chi\\ r_1(\gamma) \equiv r_1\bmod N\\ \vdots \\ r_d(\gamma)\equiv r_d \bmod N}}{\widehat{\nu}(\gamma)}\text{ and } h(\chi,r_1,\dots,r_d):=\sum_{\substack{\gamma \in \Lambda \\ \chi_\gamma=\chi\\ r_1(\gamma) \equiv r_1\bmod N\\ \vdots \\ r_d(\gamma)\equiv r_d \bmod N}}{(\widehat{\mu})_\mathbb{Z}(\gamma)}
\end{equation*}
Since $\Lambda$ is finite there is $N_0$, such that if $N \geq N_0$ then there is at most one non-zero summand in the sums above for any $(\chi ,r_1,\dots,r_d) \in\widehat{H}\times \mathbb{Z}_N^d$. In particular $d(g,\mathbb{Z}) \leq d(\widehat{\nu},\mathbb{Z}) \leq \delta(\frac{1}{3}\varepsilon_1,\varepsilon_2,2M)$, and $g_\mathbb{Z}=h$ (since we saw earlier that $(\widehat{\nu})_\mathbb{Z}=(\widehat{\mu})_\mathbb{Z}$).

The map $\Phi:H \times \mathbb{Z}_N^d \rightarrow (\widehat{H}\times \mathbb{Z}_N^d)^\wedge$ defined by
\begin{equation*}
\Phi(h,x_1,\dots,x_d) (\chi,r_1,\dots,r_d)=\chi(h)\exp\left(2\pi i \frac{x_1r_1}{N}\right)\cdots \exp\left(2\pi i \frac{x_dr_d}{N}\right)
\end{equation*}
is a bijection and
\begin{equation*}
\widehat{g}\left(\Phi\left(h,x_1,\dots,x_d\right)\right)=\frac{1}{|\widehat{H}\times \mathbb{Z}_N^d|}\sum_{\lambda \in \Lambda}{\widehat{\nu}(\lambda)\chi_\lambda(h)\exp\left(2\pi i \frac{x_1r_1(\lambda)}{N}\right)\cdots \exp\left(2\pi i \frac{x_dr_d(\lambda)}{N}\right)},
\end{equation*}
so
\begin{equation*}
\|g\|_{A(\widehat{H}\times \mathbb{Z}_N^d)}=\frac{1}{|H|N^d}\sum_{h \in H}{\sum_{x_1,\dots,x_d=1}^N{\left|\sum_{\lambda \in \Lambda}{\widehat{\nu}(\lambda) \chi_\lambda(h)\exp\left(2\pi i \frac{x_1r_1(\lambda)}{N}\right)\cdots \exp\left(2\pi i \frac{x_dr_d(\lambda)}{N}\right)}\right|}}.
\end{equation*}
The inner sum here is a $d$-dimensional Riemann sum and we have
\begin{align*}
\|g\|_{A(\widehat{H}\times \mathbb{Z}_N^d)}& \rightarrow \frac{1}{|H|}\sum_{h \in H}{\int{\left|\sum_{\lambda \in \Lambda}{\widehat{\nu}(\lambda) \chi_\lambda(h)\exp\left(2\pi i \theta_1r_1(\lambda)\right)\cdots \exp\left(2\pi i \theta_d r_d(\lambda)\right)}\right|d\theta_1\dots d\theta_d}}\\ & = \int{\left|\sum_{\lambda \in \Lambda}{\widehat{\nu}(\lambda)\tilde{\lambda}(h,\theta_1,\dots,\theta_d)}\right|d\mu_{H \times (\mathbb{R}/\mathbb{Z})^d}(h,\theta_1,\dots,\theta_d)}\\
&= \int{\left|\sum_{\lambda \in \Lambda}{\widehat{\nu}(\lambda)\tilde{\lambda}\circ \pi (z)}\right|d\mu_{G/K}(z)}= \int{\left|\sum_{\lambda \in \Lambda}{\widehat{\nu}(\lambda)\lambda(z)}\right|d\mu_{G}(z)}=\|\nu\|.
\end{align*}
Similarly $\|g_\mathbb{Z}\|_{A(\widehat{H}\times \mathbb{Z}_N^d)}\rightarrow \|\mu_\mathbb{Z}\|$.

Let $N_1 \geq N_0$ be large enough that $\|g\|_{A(\widehat{H}\times \mathbb{Z}_N^d)} \leq 2M$ for all $N \geq N_1$. Then by Theorem \ref{tws} we have $\|g_\mathbb{Z}\|_{A(\widehat{H}\times \mathbb{Z}_N^d)} \leq \left(1+\frac{1}{3}\varepsilon_1\right)\|g\|_{A(\widehat{H}\times \mathbb{Z}_N^d)}+\varepsilon_2$. Taking limits we have
\begin{equation*}
\|\mu_\mathbb{Z}\|\leq  \left(1+\frac{1}{3}\varepsilon_1\right)\|\nu\|+\varepsilon_2 \leq \left(1+\frac{1}{3}\varepsilon_1\right)^2\|\mu\| + \varepsilon_2 \leq (1+\varepsilon_1)\|\mu\|+\varepsilon_2.
\end{equation*}
The result is proved.
\end{proof}
An attempt to remove the strong continuity requirement from Theorem \ref{glow} using the presented approach fails at a very basic level -- we are not able to define the measure $\mu_\mathbb{Z}$ in the statement of Theorem \ref{skn}. 

\begin{thm}\label{glowrefined}
There is a function $\delta:(0,1] \rightarrow (0,1]$ such that if $(a_n)_{n=1}^\infty$ is a sequence of positive reals with $a_1 \leq 1$ and $a_{n+1} \leq a_n\delta(a_n)$ for all $n \geq 1$ then, for $G$ compact, and $\mu\in M(G)$ strongly continuous with $\widehat{\mu}(\widehat{G})\subset\{a_{n}:n\in\mathbb{N}\}\cup\{0\}$ and $\|\mu\|\leq 1$, we have $\mu\ast\mu\in L^{1}(G)$.
\end{thm}
\begin{proof}
Define 
\begin{equation*}
\delta(x) = \min\left\{\frac{1}{9}\exp(-18x^{-2}),\delta'\left(\frac{1}{2},\frac{1}{2},x^{-3/2}\right)\right\}
\end{equation*}
where $\delta'$ is as in Theorem \ref{skn}. Note by induction that $a_n \leq 9^{-(n-1)}$ for all $n \geq 1$.

Suppose that $\mu$ satisfies the hypotheses of the theorem and set
\begin{equation*}
A_{n}=\{\gamma\in\widehat{G}:\widehat{\mu}(\gamma)=a_{n}\}.
\end{equation*}
Of course, $A_{j}\cap A_{k}=\emptyset$ for $j\neq k$. Apply Theorem \ref{pom} with $r =  \lceil17a_n^{-2}\rceil$ to the measure $a_n^{-1}\mu$ which is strongly continuous and has $\|a_n^{-1}\mu\| <\frac{1}{4}\sqrt{r}$ and $|a_n^{-1}\widehat{\mu}(\gamma)| \leq \exp(-r)$ for all $\gamma \not \in \bigcup_{k=1}^n{A_k}$. This tells us that $\bigcup_{k=1}^n{A_k}$ is finite.

Define $\mu_n$ by Fourier inversion such that $\widehat{\mu_n}=1_{A_n}$ and $\tau_n \in M(G)$ by
\begin{equation*}
\tau_n:=\mu-\sum_{k=1}^{n-1}{a_{k}\mu_k}.
\end{equation*}
This rearranges to
\begin{equation}\label{tozs}
   a_n^{-1}\tau_{n}=\mu_{n}+a_n^{-1}\tau_{n+1}.
\end{equation}

We shall show by induction that $\|\tau_n\| \leq a_n^{-1/2}$ and $\|\mu_{n-1}\|\leq \frac{3}{2}a_{n-1}^{-3/2}+\frac{1}{2}$ (with $\mu_0$ the zero measure). This is certainly true for $n=1$; now suppose we are at step $n$ of the induction. Since $\widehat{\mu_n}$ takes the values $0$ and $1$, and $|\widehat{\tau}_{n+1}(\gamma)|\leq a_{n+1}$ for all $\gamma \in \widehat{G}$, we have from (\ref{tozs}) that $d\left(a_n^{-1}\widehat{\tau_{n}},\mathbb{Z}\right) \leq \frac{a_{n+1}}{a_{n}}$ and $\left(a_n^{-1}\widehat{\tau_{n}}\right)_{\mathbb{Z}}=\widehat{\mu_n}$. Apply Theorem \ref{skn} with $\varepsilon_1=\varepsilon_2=\frac{1}{2}$ and $M=a_n^{-3/2}$ to $a_n^{-1}\tau_n$ which certainly has real Fourier-Stieltjes transform and $\|a_n^{-1}\tau_n\| \leq a_n^{-3/2}$, to get $ \|\mu_n\|\leq \frac{3}{2}a_n^{-3/2} +\frac{1}{2}$. This is the second part of the inductive hypothesis for $n+1$. But it also implies
\begin{equation*}
\|\tau_{n+1}\|\leq \|\tau_{n}\|+a_{n}\|\mu_n\|\leq a_n^{-1/2} + \frac{3}{2}a_n^{-1/2} + \frac{1}{2}a_n \leq 3a_n^{-1/2}\leq a_{n+1}^{-1/2}
\end{equation*}
as required.

Finally, let $\rho_n:=\sum_{k=1}^n{a_k^2\mu_k}$.  Then
\begin{equation}\label{cauchy2}
\|\rho_n-\rho_m\| \leq \sum_{k=m+1}^n{a_k^2\|\mu_k\|} \leq \sum_{k=m+1}^n{2a_k^{1/2}} \leq 2\sum_{k=m+1}^n{3^{-k}} \leq 3^{-m} \text{ for }n>m.
\end{equation}
Hence $(\rho_n)_n$ is Cauchy and so converges. Since their Fourier-Stieltjes transforms have finite support, the measures $\rho_n$ are also in $L^1(G)$, which is a closed ideal in $M(G)$ and $M(G)$ is complete, so $\rho_n \rightarrow \rho$ for some $\rho \in L^1(G)$.

Since the $A_k$s are disjoint we have $\widehat{\rho_n}(\gamma) \rightarrow \widehat{\mu}(\gamma)^2$ for $\gamma \in \widehat{G}$ and hence $\widehat{\rho}(\gamma)=\widehat{\mu}(\gamma)^2$ for all $\gamma \in \widehat{G}$, and hence by the uniqueness of the Fourier transform $\rho = \mu \ast \mu$ and the result is proved since $\rho \in L^1(G)$.
\end{proof}

\end{document}